\documentclass[a4,10pt]{amsart}
\usepackage{amsmath}
\usepackage{amsthm}
\usepackage{amsfonts}
\usepackage{amssymb}
\usepackage{graphicx}
\usepackage{epstopdf}
\usepackage{amscd}
\usepackage{color}
\usepackage{enumerate}
\usepackage{hyperref}
\usepackage[hyperpageref]{backref}
\theoremstyle{plain}
\numberwithin{equation}{section}
\newtheorem{theorem}{Theorem}[section]
\newtheorem{prop}[theorem]{Proposition}
\newtheorem{lem}[theorem]{Lemma}
\newtheorem{rmk}[theorem]{Remark}

\newtheorem{defn}[theorem]{Definition}
\theoremstyle{definition}

\def\C{\mathbf {C}}

\renewcommand{\i}{{\rm i}}

\makeatletter
\def\cleardoublepage{\clearpage\if@twoside \ifodd\c@page\else
\hbox{}
\thispagestyle{empty}
\newpage
\if@twocolumn\hbox{}\newpage\fi\fi\fi}
\makeatother
\title{Maz'ya-Shaposhnikova formula in Magnetic Fractional Orlicz-Sobolev spaces}

\author{A. Maione}
\address{Alberto Maione: Dipartimento di Matematica\\Universit\`a di Trento\\ Via Sommarive 14\\ 38123, Povo (Trento) - Italy\\}
\email{alberto.maione@unitn.it}
\author{A. M. Salort}
\address{Ariel M. Salort: Departamento de Matem\'atica, FCEyN - Universidad de Buenos Aires and IMAS - CONICET Ciudad Universitaria, Pabell\'on I (1428) Av. Cantilo s/n., Buenos Aires, Argentina.}
\email{asalort@dm.uba.ar}
\urladdr{http://mate.dm.uba.ar/~asalort}
\author{E. Vecchi}
\address{Eugenio Vecchi: Dipartimento di Matematica\\Politecnico di Milano\\ Piazza Leonardo da Vinci\\ 20133, Milano (MI) - Italy\\}
\email{eugenio.vecchi@polimi.it}
\thanks{A.M. and E.V. are supported by the Indam-GNAMPA project 2020 ``Convergenze variazionali per funzionali e operatori dipendenti da campi vettoriali\,''. A.M. is also supported by MIUR and University of Trento, Italy.
A.M.S. is partially supported by grants UBACyT 20020130100283BA, CONICET PIP 11220150100032CO and ANPCyT PICT 2012-0153.}

\date{\today}
\begin{document}
\begin{abstract}
In this note we prove the validity of the Maz'ya-Shaposhnikova formula 
in magnetic fractional Orlicz-Sobolev spaces. This complements a previous
study of the limit as $s \uparrow 1$ performed by the second author in 
\cite{BS2}.
\end{abstract}
\maketitle
\section{Introduction}


The Hamiltonian formulation of classical mechanics leads to 
an expression for the energy $H$ of a particle of mass $m$, given by
$$H(x,p) = \tfrac{p^2}{2m} + V(x), \quad (x,p) \in \mathbb{R}^n \times \mathbb{R}^n,$$
\noindent where $p$ is the momentum, $V$ is the potential energy
and $T(p):=\tfrac{p^2}{2m}$ is the kinetic energy. 
Restricting our attention to the kinetic part of the energy $H$, if we are
interested in the study of the motion of an electrically charged 
particle under the influence of a magnetic field $B$, the kinetic
energy has to be modified as follows (see e.g. \cite{LS}):
$$T(x,p):= \tfrac{1}{2m}(p + \tfrac{e}{c} A(x))^2,$$
where $e$ is the charge and $c$ is the speed of light.
The passage from classical mechanics to quantum mechanics requires the 
introduction of suitable operators. The quantization is performed by 
replacing the momentum $p$ by the Euclidean gradient $\nabla$, at least when we do
not take care of the presence of an external magnetic field $B$, otherwise
it is customary to replace the classical gradient $\nabla$ by
the so-called {\it magnetic gradient} $\nabla_A := \nabla - \i A$,
where $A:\mathbb{R}^n \to \mathbb{R}^n$ is a vector potential and ``$\i$\,'' denotes the imaginary unit. 
In particular, in the physical case of $\mathbb{R}^3$, if $A$ is smooth enough,
it holds that $\mathrm{curl}(A) =B$. The magnetic gradient can be used to define a second order differential operator usually called the {\it magnetic Laplacian}. See for instance \cite{EL, JR, JR2} and the references therein.
From the mathematical point of view, this first
order differential operator leads to the definition
of the {\it magnetic Sobolev space} $H^{1}_{A}(\mathbb{R}^n;\mathbb{C})$ (see e.g. \cite[Chapter 7, Section 19]{LossLieb}),
as the space of complex valued functions $u \in L^{2}(\mathbb{R}^n;\mathbb{C})$
such that $\nabla_A u \in (L^{2}(\mathbb{R}^n;\mathbb{C}))^n$. In order to have a proper
definition it is necessary to ask for some regularity on the
vector potential $A$, for example $L^{2}_{loc}(\mathbb{R}^n)$. 
With standard modifications one can define the magnetic Sobolev space $H^{1}_{A}(\Omega;\mathbb{C})$,
where $\Omega \subset \mathbb{R}^n$ is an open set. The Hilbertian case
(i.e. $p=2$) is certainly the physically relevant one, but it is also
possible to define magnetic Sobolev spaces $W^{1,p}_A(\mathbb{R}^n;\mathbb{C})$ 
(or $W_{A}^{1,p}(\Omega;\mathbb{C})$) for every $p \geq 1$, see e.g. \cite{PSV2}.\\
In the last years, there has been a certain interest towards a characterization
of magnetic Sobolev spaces (see \cite{NPSV,NPSV2,PSV2,SqVo}), including the validity
of a Maz'ya-Shaposhnikova formula in this setting, see \cite{PSV}.
In the classical diamagnetic case (i.e. when $A \equiv 0$), in \cite{MS} Maz'ya and Shaposhnikova showed that the fractional 
Sobolev seminorm recovers the $L^p$-norm. To be more precise, they proved that
for every $u \in \bigcup_{s\in (0,1)}W^{s,p}_{0}(\mathbb{R}^n)$, 
 \begin{equation}\label{eq:MS}
     \lim_{s\downarrow 0} s\iint_{\mathbb{R}^{2n}}\frac{|u(x)-u(y)|^p}{|x-y|^{n+sp}}\,dx\,dy= 
     2\dfrac{|S|^{n-1}}{p}\int_{\mathbb{R}^n}|u|^p\,dx,
 \end{equation}
 \noindent where $|S|^{n-1}$ denotes the measure of the unit sphere. 
 We point out that the function space $W^{s,p}_{0}(\mathbb{R}^n)$ 
stands for the closure of $C^{\infty}_{0}(\mathbb{R}^{n})$ with
 respect to the fractional Gagliardo seminorm 
 $$[u]_{s,p}:=\left(\iint_{\mathbb{R}^n \times \mathbb{R}^n}\frac{|u(x)-u(y)|^p}{|x-y|^{n+sp}}\,dx\,dy\right)^{1/p}.$$
 \noindent We stress that \eqref{eq:MS} has been the object of other generalizations
oriented to cover the case of ambient spaces different from ours, see e.g. \cite{BGT,Milman, Triebel}.
We also recall that \eqref{eq:MS} can be considered as the natural counterpart of the limit as $s \uparrow 1$
studied by Bourgain, Brezis and Mironescu in \cite{BBM}. 
\medskip

\indent The extension of \eqref{eq:MS} to the magnetic setting, requires the introduction
of appropriate {\it magnetic fractional Sobolev spaces} $W_{A}^{s,p}(\mathbb{R}^n;\mathbb{C})$.
In particular, when $p=2$, it is possible to define the fractional Sobolev
space $H^{s}_{A}(\Omega;\mathbb{C})$ for a given open set $\Omega \subset \mathbb{R}^{n}$ as the space
of complex-valued functions $u \in L^{2}(\Omega;\mathbb{C})$ for which
the following magnetic Gagliardo seminorm $[u]_{H^{s}_{A}(\Omega;\mathbb{C})}$ is finite:
$$
[u]_{H^{s}_A(\Omega;\mathbb{C})}:=\left(\iint_{\Omega\times\Omega}\dfrac{|u(x)-e^{\i (x-y)\cdot A\left(\frac{x+y}{2}\right)}u(y)|^2}{|x-y|^{n+2s}}dxdy\right)^{1/2}.
$$
This space boils down to the classical fractional Sobolev space when $A\equiv 0$ (and $u$ is real-valued) and has been proved
to be the correct space where to study nonlocal problems driven by the so-called {\it magnetic fractional Laplacian} 
$(-\Delta)^{s}_{A}$. Without any attempt of completeness, 
we refer to \cite{  AmbrosioCPDE, AmbrosioDCDS, AmbrosioProc,AmdA,  BSX,AvSq, FPV17, FV,  LRZ18,MPSZ,  MRZ}.
The nonlocal operator $(-\Delta)^{s}_{A}$ via {\it midpoint prescription} has been introduced in \cite{AvSq} for any $s \in (0,1)$ and
it admits the following representation
when acting on smooth complex-valued functions 
$u\in C^{\infty}_0(\mathbb{R}^n;\mathbb{C})$
\begin{equation}\label{eq:LaplA}
(-\Delta)_{A}^{s}u(x)
=2\lim_{\varepsilon\to 0^+}\int_{\mathbb{R}^n\setminus B(x,\varepsilon)}\frac{u(x)-e^{\i(x-y)\cdot A\left(\frac{x+y}{2}\right)}u(y)}{|x-y|^{n+2s}} dy, \qquad x\in\mathbb{R}^n,
\end{equation}
\noindent  where $B(x,\varepsilon)$ denotes the ball of center $x$ and radius $\varepsilon$. It is easy to notice
that when $A \equiv 0$ and $u$ is real-valued, $(-\Delta)_{A}^{s}u$ coincides with the usual definition of fractional Laplacian.
See Remark \ref{MagneticRemark} and Remark \ref{GaugeRemark} 
for further comments on the magnetic fractional Laplacian and magnetic fractional spaces.

\medskip

The aim of this note is to continue the study of properties of magnetic spaces, focusing on the 
extension of \eqref{eq:MS} to the case of magnetic fractional Orlicz-Sobolev spaces,
a class of function spaces which has been recently introduced in \cite{BS2}, 
where the authors proved the validity of a Bourgain-Brezis-Mironescu-type 
formula, so extending to the magnetic setting their previous result in \cite{BS}.
In order to properly state our result, let us first give a basic glance to the objects
in play. We consider a smooth enough magnetic potential $A\colon \mathbb{R}^n\to \mathbb{R}^n$ 
and a Young function $G$ satisfying the following standard growth condition:
\begin{equation} \label{L} \tag{L}
1\leq p^-\leq \frac{t\,G'(t)}{G(t)} \leq p^+<\infty \quad \text{for every } t>0.
\end{equation}
We also consider the Young function 
$\overline{G} \colon \mathbb{R}^{+}_0 \to \mathbb{R}^{+}_0$ 
naturally associated to $G$, which is defined as 
\begin{equation}\label{bar}
\overline{G}(t):= \int_0^tG(\tau) \frac{d\tau}{\tau}, \quad t\in\mathbb{R}^{+}_0.
\end{equation}

For any $s \in (0,1)$ it is now possible to introduce a proper notion of
{\em magnetic fractional Orlicz Sobolev spaces} $W_{A}^{s,G}$, 
see section \ref{section2} for the precise definition. 

Our main result reads as follows. 

\begin{theorem}\label{MS1}
Let $G$ be a Young function satisfying \eqref{L} and $A\colon \mathbb{R}^n \to  \mathbb{R}^n$ be a magnetic field. 
If $u\in\bigcup_{s\in (0,1)}W_{A,0}^{s,G}(\mathbb{R}^n;\mathbb{C})$, then
\begin{equation}\label{MainRes}
\lim_{s\downarrow 0}s \iint_{\mathbb{R}^n\times\mathbb{R}^n}G\left(\left| \frac{u(x)-e^{\i(x-y) A\left(\tfrac{x+y}{2}\right)} u(y)}{|x-y|^s} \right|\right)\,\dfrac{dx\, dy}{|x-y|^{n}} 
=\dfrac{2 \, \omega_n}{n} \int_{\mathbb{R}^n}\overline{G}(|u(x)|)\,dx.
\end{equation}
\end{theorem}

The proof of Theorem \ref{MS1} is presented in section \ref{section4} and it is obtained combining careful estimates
of both the liminf and the limsup. The technique adopted to achieve those estimates follows the lines
of the original proof by Maz'ya and Shaposhnikova in \cite{MS}.
We mention that the same scheme has been already successfully applied
in providing extensions of \eqref{eq:MS} in different settings: we refer to \cite{PSV} for the classical fractional magnetic
setting, \cite{ACPS} for the case of fractional Orlicz-Sobolev spaces and \cite{CMSV,CMSV2} for the case
of fractional Orlicz-Sobolev spaces in Carnot groups. \\
\noindent We note that, coherently with \cite{PSV}, \eqref{MainRes} shows that the limit as $s \downarrow 0$ kills the magnetic
effect so recovering the same limit found in \cite{ACPS}. We also stress that assuming condition \eqref{L} on the Young function $G$
is not restrictive at all, as showed in \cite[Theorem 1.2]{ACPS}.
It is noteworthy to mention that, in order to perform the estimates necessary to prove \eqref{MainRes}, we have
to ensure the finiteness of certain integrals (see section \ref{section4} for the details). To this aim, in analogy with \cite{ACPS}, we have to prove a magnetic Hardy-type inequality for Orlicz functions for \emph{small} values of the fractional parameter $s$ (we refer to Remark \ref{rmkhardy} for more details). 
This is the content of the following
\begin{theorem} \label{hardy}
Given a Young function $G$ satisfying \eqref{L} and $s\in (0,1)$ such that $s<\frac{n}{p^+}$, there exists a constant $C=C(n,s,p^\pm)$ such that
\begin{equation}\label{hi}
   \int_{\mathbb{R}^n}G\left(\frac{|u(x)|}{|x|^{s}}\right)\,dx\leq C \iint_{\mathbb{R}^{2n}} G\left(\left|\dfrac{u(x)-e^{\i(x-y) A\left(\frac{x+y}{2}\right)} u(y)}{|x-y|^s} \right| \right) \,\dfrac{dx\, dy}{|x-y|^{n}}
\end{equation}
for all $u\in W^{s,G}_{A,0}(\mathbb{R}^n;\mathbb{C})$.
\end{theorem}
This result recovers the fractional magnetic Hardy-inequality proved in \cite{PSV} when $G(t)=t^p$
and \cite[Theorem 5.1]{ACPS2}, if instead we consider the diamagnetic case, i.e. when $A \equiv 0$.
The proof is an easy combination of the latter with a standard diamagnetic inequality, see section \ref{section3} for the details.

\medskip

We want to stress that the interest towards fractional Orlicz-Sobolev spaces is recently moving to the
study of nonlocal PDEs driven by the so-called $G$-Laplacian, which is a nonlocal operator
naturally modeled on a given Young function $G$, see e.g. \cite{  BBX, BOT,DNFBS,BLS, Salort}. We refer also to \cite{BBR, MiR, MRR} for some related problems concerning the local $G$-Laplacian.

\medskip

\indent The paper is organized as follows: in section \ref{section2} we provide the necessary definitions and basic results
concerning magnetic fractional Orlicz-Sobolev spaces, in section \ref{section3} we prove Theorem \ref{hardy} and finally
in section \ref{section4} we prove Theorem \ref{MS1}.

\section{Preliminaries}\label{section2}

In this section we introduce definitions and notations and  we recall few results needed in the next sections.

\subsection{Young functions}
We recall now the basic notions concerning Young functions. We refer the interested reader to the books \cite{DHHR,PKJF} for a comprehensive introduction to the subject.
\begin{defn} Let $\phi:\mathbb{R}^+_0\to\mathbb{R}^+_0$ be a function satisfying the following conditions:
\begin{itemize} 
    \item[$(i)$] $\phi(0)=0$, and $\phi(t)>0$ for any $t>0$;
    \item[$(ii)$] $\phi$ is non-decreasing;
    \item[$(iii)$] $\phi$ is right-continuous and $\lim_{t\to+\infty}\phi(t)=+\infty$.
\end{itemize}
We call Young function the real valued function $G: \mathbb{R}^+_0 \to \mathbb{R}^+_0$ given by 
$$G(t)=\int_0^t \phi(s)\,ds.$$
\end{defn}
We recall that properties $(i)$ to $(iii)$ imply that every Young function $G$ is continuous, 
locally Lipschitz continuous, strictly increasing and convex on $\mathbb{R}^+_0$. 
It is also quite easy to see that $G(0)=0$ and that $G$ is superlinear both at zero and at $+\infty$.
We can also assume without loss of generality that $G(1)=1$. 
Finally, the inverse $G^{-1}:\mathbb{R}^+_0\to\mathbb{R}^+_0$ of $G$ is also well defined and inherits a 
bunch of properties: in particular, $G^{-1}$ is continuous, concave, strictly increasing,
$G^{-1}(0)=0$ and $G^{-1}(1)=1$.\\
\noindent As already stated in the introduction, we will work with Young functions satisfying \eqref{L}.
As a direct consequence, we notice that
\begin{align*}
  &s^{\overline{p}}G(t)\leq G(st)\leq s^{\tilde{p}}G(t),\tag{$G_1$}\label{G1}\\
  &G(s+t)\leq\frac{2^{p^+}}{2}(G(s)+G(t))\tag{$G_2$}\label{G2}
\end{align*}
for any $s,t\in\mathbb{R}^+_0$, where $s^{\tilde{p}}:=\max\{s^{p^-},s^{p^+}\}$ and $s^{\overline{p}}:=\min\{s^{p^-},s^{p^+}\}$.\\

Young functions aim at generalizing the power functions, hence $G(t)=t^p$, for a certain $p\geq 1$, is definitely
a trivial example of a Young function for which $p^{-} = p = p^{+}$. There are of course less trivial instances of Young functions e.g. 
logarithmic perturbation of powers: given positive constants $a,b,c>0$ one considers a function $G$ such that $G'(t)=t^a \log(b+ct)$. 
In this case we have that $p^-=1+a$ and $p^+=2+a$.

We want to recall that condition \eqref{L} is equivalent to another
condition usually assumed when dealing with Young functions, namely the
$\Delta_2$-condition.

\begin{defn}\label{delta2}
Let $G$ be a Young function. We say that $G$ satisfies the $\Delta_2$-condition if
there exists a positive constant $\C>2$ such that
\begin{equation*}
G(2t)\leq\C\,G(t)\quad\text{for every }t\in\mathbb{R}^+_0.
\end{equation*}
\end{defn}
We close this brief recap noticing that $\overline{G}$ 
is a bounded Young function. 
Keeping in mind its definition \eqref{bar}, the boundedness follows from the monotonicity of both
$G(t)$ and $\tfrac{G(t)}{t}$. We also stress that the functions $G$ and $\overline{G}$ are {\it equivalent} 
Young functions, in the sense that
\begin{equation}\label{equivalent}
    G\left(\frac{t}{2}\right)\leq\overline{G}(t)\leq G(t)\quad\text{for every }t\in\mathbb{R}_0^+.
\end{equation}
\medskip

\subsection{Magnetic spaces}
We are now ready to introduce the relevant functional spaces appearing in Theorem \ref{MS1}. \\
To simplify the readability, we introduce the following compact notation:
let $s \in (0,1)$, we denote the magnetic H\"{o}lder quotient of order $s$ as
\begin{equation}\label{DsA}
D_s^A u(x,y) := \frac{u(x)-e^{\i(x-y) A\left(\tfrac{x+y}{2}\right)} u(y)}{|x-y|^s}.
\end{equation}
We will also denote
$$d\mu(x,y):=\frac{dx \,dy}{|x-y|^n}.$$
\noindent We notice that when $A\equiv 0$, $D_s^0 u = D_s u := \dfrac{u(x)-u(y)}{|x-y|^s}$ 
is the usual $s$-H\"older quotient appearing in the definition of fractional Sobolev spaces.

\begin{defn}
Let $G$ be an Orlicz function, $s\in (0,1)$ be a fractional parameter and $A:\mathbb{R}^n\to\mathbb{R}^n$
a smooth enough vector potential.
We define the spaces $L^G(\mathbb{R}^n;\mathbb{C})$ and $W^{s,G}_A(\mathbb{R}^n;\mathbb{C})$ as follows:
\begin{align*}
&L^G(\mathbb{R}^n;\mathbb{C}) :=\left\{ u\colon \mathbb{R}^n \to \mathbb{C} \text{ measurable}\colon  I_{G}(u) < \infty \right\},\\
&W^{s,G}_A(\mathbb{R}^n;\mathbb{C}):=\left\{ u\in L^G(\mathbb{R}^n;\mathbb{C})\colon I_{s,G}^A(u)<\infty \right\},
\end{align*}
where $I_G$ and $I_{s,G}^A$ are defined as
$$
I_{G}(u) := \int_{\mathbb{R}^n}G(|u(x)|)\,dx
$$
and
\begin{equation}\label{IsGA}
I_{s,G}^A(u) := \iint_{\mathbb{R}^{2n}}G\left(|D_s^A u(x,y)|\right)\,d\mu.
\end{equation}
\end{defn}
These spaces become Banach spaces when endowed with the so-called Luxemburg norms defined through $I_{s, G}^A$, namely
$$
\|u\|_{s, G}^A = \|u\|_G + |u|_{s,G}^A, 
$$
where 
$$
\|u\|_G := \inf\{\lambda>0\colon I_G(\tfrac{u}{\lambda})\le 1\}
$$
is the usual (Luxemburg) norm on $L^G(\mathbb{R}^n;\mathbb{C})$ and
$$
|u|_{s,G}^A := \inf\{\lambda>0\colon I_{s,G}^A(\tfrac{u}{\lambda})\le 1\}.
$$
Finally, we define the fractional Magnetic Orlicz-Sobolev space $W^{s,G}_{A,0}(\mathbb{R}^n;\mathbb{C})$ 
as the closure of $C^{\infty}_{0}(\mathbb{R}^{n}; \mathbb{C})$ with respect to the magnetic fractional Orlicz seminorm $|u|_{s,G}^A$.
We note that when $G(t)=t^p$ we recover the magnetic fractional Sobolev spaces defined in \cite{PSV,PSV2,SqVo}.
At the same time, if we assume $A \equiv 0$ the above definitions lead to the fractional Orlicz-Sobolev spaces considered
in \cite{ACPS,ACPS2}. Combining the last observations, it is also obvious that for $G(t)=t^p$ and $A\equiv 0$ we recover
the classical fractional Sobolev spaces.

\medskip
\begin{rmk}\label{MagneticRemark}
In the definition of the H\"{o}lder quotient of order $s$, and hence
in the definition of the fractional Orlicz-Sobolev spaces, we actually chose
the so-called midpoint prescription
$$(x,y)\mapsto A\left(\dfrac{x+y}{2}\right),$$
\noindent which is closely related to the magnetic fractional Laplacian defined in \eqref{eq:LaplA}.
It is noteworthy to mention that actually, for $s=\tfrac{1}{2}$, the definition of the fractional operator $(-\Delta)^s_A$
dates back to the '80s, and it is closely related to the proper definition of
a quantized operator corresponding to the symbol of the classical relativistic
Hamiltonian, namely
$$\sqrt{(\xi - A(x))^2 + m^2} + V(x), \quad (\xi,x) \in \mathbb{R}^n \times \mathbb{R}^n.$$
In particular, it is related to the {\it kinetic part} of the above
symbol. In the survey \cite{I13} it is explained that
there are at least three definitions for
such a quantized operator appearing in the literature:
two of them are given in terms of pseudo-differential operators, while the third one as
the square root of a suitable non-negative operator. 
In \cite{I13}, Ichinose showed that these three non--local operators  
are in general not the same, but they do coincide whenever dealing with a linear vector
potential $A$. A well studied example of a linear potential $A$ is the so-called
Ahronov-Bohm potential. We finally notice that, in the physically relevant case of $\mathbb{R}^{3}$, 
the linearity of $A$ is actually equivalent to require a constant magnetic field. 
\end{rmk}

\begin{rmk}\label{GaugeRemark}
As mentioned before, one may then replace midpoint prescription
with other prescriptions, e.g. the averaged one
$$
(x,y)\mapsto \int_0^1A\left((1-\vartheta)x+\vartheta y\right)d\vartheta=:
A_\sharp(x,y).
$$
From the physical point of view, the latter has the advantage that the magnetic fractional Laplacian associated to it, 
i.e. $(-\Delta)^s_{A_\sharp}$, turns out to be Gauge covariant (see e.g. \cite[Proposition 2.8]{I13}), namely 
$$
(-\Delta)^s_{(A+\nabla\phi)_\sharp}=e^{\i \phi}(-\Delta)^s_{A_\sharp}e^{-\i \phi}.
$$ 
It is clear that one can define fractional Orlicz-Sobolev spaces according to the averaged prescription.
All the results presented in this note remain true even for these spaces, once the obvious modifications are made.
\end{rmk}


\section{A fractional Magnetic Hardy-type inequality}\label{section3}
In this section we prove Theorem \ref{hardy}. The proof
comes out as a combination of the Hardy-type inequality proved in \cite{ACPS2} with
the fractional diamagnetic inequality, which in turn heavily relies upon 
the so-called {\it diamagnetic inequality}. 
The latter is well-known in the classical setting, see e.g. \cite[Theorem 7.21]{LossLieb},
and it reads as follows:

\begin{prop}
Let $A\colon \Omega\to\mathbb{R}^n$ be a measurable magnetic potential 
such that $|A|<\infty$ a.e. in $\Omega$ and let $u\in W^{1,1}_\text{loc}(\mathbb{R}^n;\mathbb{C})$. 
Then 
\begin{equation}\label{DI}
|\nabla |u|(x)| \le |\nabla u(x) - \i A(x) u(x)|,
\end{equation}
for a.e. $x\in\Omega$.
\end{prop}

The fractional analogue of \eqref{DI} was provided in \cite[Lemma 3.1 and Remark 3.2]{AvSq}:
\begin{prop} \label{diamagnetic}
Let $A\colon\mathbb{R}^n\to\mathbb{R}^n$ be a measurable magnetic potential such that $|A|<\infty$ a.e. in $\mathbb{R}^n$
and let $u\colon\mathbb{R}^n\to\mathbb{C}$ be a measurable function such that $|u|<\infty$ a.e. in $\mathbb{R}^n$. 
Then 
\begin{equation}\label{FDI}
||u(x)|-|u(y)||\le \left|u(x) - e^{\i(x-y)A\left(\tfrac{x+y}{2}\right)}u(y)\right|,
\end{equation}
for a.e. $x,y\in\mathbb{R}^n$.
\end{prop}

\begin{rmk} \label{diamagnetic.rem}
Observe that, using the compact notation introduced before, 
the fractional diamagnetic inequality \eqref{FDI} can be re-stated as
\begin{equation}\label{FDI2}
|D_s|u|(x,y)|\le |D_s^A u(x,y)|,
\end{equation}
a.e. $x, y\in\mathbb{R}^n$, where $D_s v(x,y) = \frac{v(x)-v(y)}{|x-y|^s}$.
\end{rmk}

With this at hand, we are now ready to prove Theorem \ref{hardy}.

\begin{proof}[Proof of Theorem \ref{hardy}]
Given a Young function $G$, by \cite[Theorem 5.1]{ACPS2} there exist a Young function $\hat G$ 
and a positive constant $C=C(n,s)>0$ such that
\begin{equation}\label{Cianchi}
   \int_{\mathbb{R}^n}\hat G\left(\frac{|u(x)|}{|x|^{s}}\right)\,dx
    \leq (1-s)\iint_{\mathbb{R}^{2n}}G(C |D_s u (x,y)| )\,d\mu.
\end{equation}
Moreover, in light of \cite[Propositions 5.1 and 5.2]{C}, since $s<\frac{n}{p^+}$, $\hat G$ and $G$ are equivalent Young functions, i.e., there exist two positive constants $c_1<c_2$ such that $G(c_1 t)\leq \hat G(t)\leq G(c_2 t)$ for all $t\in \mathbb{R}_0^+$. 
Therefore, by using inequality \eqref{Cianchi} on $|u|$ and the Diamagnetic inequality \eqref{FDI}, it holds that
\begin{align*}
    \int_{\mathbb{R}^n}G\left(c_1 \frac{|u(x)|}{|x|^{s}}\right)\,dx
    &\leq (1-s)\iint_{\mathbb{R}^{2n}} G\left(C|D_s|u|(x,y)|\right)\,d\mu\\
    &\leq (1-s) \iint_{\mathbb{R}^{2n}} G\left(C |D_s^A u (x,y)| \right) \,d\mu.
\end{align*}
Thus, the result follows by \eqref{G1}.
\end{proof}
\begin{rmk}\label{rmkhardy}
    Let us notice that, accordingly to the Hardy-type inequality proved by Maz'ya-Shaposhnikova, see \cite[Theorem 2]{MS}, and to the analogous result holding in the case of Orlicz spaces, see \cite{ACPS2}, the inequality \eqref{hi} holds just for small values of the fractional parameter $s$, meaning that $s<\frac{n}{p^+}$. 
        Here, the term $p^+$, which is defined in \eqref{L}, is bigger or equal to the upper Matuszewska-Orlicz index $I(G)$, taken into account in \cite{ACPS}.
        We stress that, if one is interested in proving a Hardy-type inequality for {\it all} $s \in (0,1)$ one has to reinforce the assumptions on the Orlicz function, see e.g. \cite[Section 5]{ACPS2}. On the other hand, as explained in \cite{ACPS}, the validity of condition \eqref{L}, ensures that for small enough $s$ such assumptions are satisfied by all Orlicz functions.
    Finally, we notice that, as a consequence of Theorem \ref{hardy} and \eqref{equivalent}, it holds that
    \begin{equation}\label{overlinehardy}
        \int_{\mathbb{R}^n}\overline{G}\left(\alpha\frac{|u(x)|}{|x|^{s}}\right)\,dx\leq
        \int_{\mathbb{R}^n}G\left(\alpha\frac{|u(x)|}{|x|^{s}}\right)\,dx<\infty
    \end{equation}
    for any $u\in W^{s,G}_{A,0}(\mathbb{R}^n;\mathbb{C})$ and for any $\alpha>0$.
\end{rmk}
\section{Proof of Theorem \ref{MS1}}\label{section4}
In this section we prove Theorem \ref{MS1} combining two estimates for the liminf and the limsup respectively.
\begin{lem}[Liminf estimate]\label{thliminf}
Let $u\in \cup_{s\in(0,1)} W^{s,G}_{A,0}(\mathbb{R}^n;\mathbb{C})$. Then
$$
\liminf_{s\downarrow 0} s \, I_{s,G}^A(u) \geq \frac{2\omega_n}{n}\, I_{\overline{G}}(|u|).
$$
\end{lem}
\begin{proof}
If $\liminf_{s\downarrow 0} sI_{s,G}^A(u)=+\infty$ the result follows. Otherwise, there exists a sequence $\{s_k\}_{k\in\mathbb{N}}\subset (0,1)$ such that $s_k\downarrow 0$ and
$$
\liminf_{s\downarrow 0} s\,I_{s,G}^A(u) = \lim_{k\to\infty} s_k\,I_{s_k,G}^A(u).
$$
By using Remark \ref{diamagnetic.rem} and the monotonicity of $G$ we get
$$
s_k\iint_{\mathbb{R}^{2n}}  G(|D_{s_k}|u||)\,d\mu\le s_k \iint_{\mathbb{R}^{2n}} G(|D_{s_k}^A u|)\,d\mu
$$
that is, for any $k\in\mathbb{N}$,
$$
s_k \, I_{s_k,G}(|u|) \leq  s_k \, I_{s_k,G}^A(u).
$$
Thus, taking the limit as $k\to\infty$ in the last inequality and applying \cite[Theorem 1.1]{ACPS} to $|u|$, we get
$$
\frac{2\omega_n}{n}\, I_{\overline{G}}(|u|) \leq \lim_{s_k\to\infty} s_k \, I_{s_k,G}^A(u),
$$
which concludes the proof.
\end{proof}

We can now move to the upper estimate for the limsup.
\begin{lem}[Limsup estimate]\label{thlimsup}
Let $u\in \cup_{s\in(0,1)} W^{s,G}_{A,0}(\mathbb{R}^n;\mathbb{C})$. Then
$$
\limsup_{s\downarrow 0} s \, I_{s,G}^A(u) \leq  \frac{2\omega_n}{n}\, I_{\overline{G}}(|u|).
$$
\end{lem}

\begin{proof}
First observe that by Fubini's Theorem and a change of variables
\begin{align*}
&\int_{\mathbb{R}^n} \left( \int_{\{|y|<|x|\}} G\left(\left| \frac{u(x)-e^{i(x-y) A\left(\frac{x+y}{2}\right)} u(y)}{|x-y|^s} \right|\right)\frac{dy}{|x-y|^n}\right)dx\\
&=
\int_{\mathbb{R}^n} \left( \int_{\{|x|>|y|\}} G\left(\left| \frac{u(x)-e^{i(x-y) A\left(\frac{x+y}{2}\right)} u(y)}{|x-y|^s} \right|\right)\frac{dx}{|x-y|^n}\right)dy\\
&=
\int_{\mathbb{R}^n} \left( \int_{\{|y|>|x|\}} G\left(\left| \frac{u(y)-e^{-i(x-y) A\left(\frac{x+y}{2}\right)} u(x)}{|x-y|^s} \right|\right)\frac{dy}{|x-y|^n}\right)dx.
\end{align*}
Since
$$
\left| \left( \frac{u(y)-e^{-i(x-y) A\left(\frac{x+y}{2}\right)} u(x)}{|x-y|^s}\right) \frac{e^{i(x-y) A\left(\frac{x+y}{2}\right)}}{e^{i(x-y) A\left(\frac{x+y}{2}\right)}} \right| = \left|   \frac{u(x)-e^{i(x-y) A\left(\frac{x+y}{2}\right)} u(y)}{|x-y|^s}   \right|,
$$
then
\begin{align*}
s \, I_{s,G}^A(u) &= s \, \int_{\mathbb{R}^n}\left(\int_{\{|y|\geq |x|\}} G(|D_s^A u(x,y)|)\frac{dy}{|x-y|^n}\right)dx\\&\quad + s \, \int_{\mathbb{R}^n}\left(\int_{\{|y|< |x|\}} G(|D_s^A u(x,y)|)\frac{dy}{|x-y|^n}\right)dx\\
&=2s  \, \int_{\mathbb{R}^n} \left(\int_{\{|y|\geq |x|\}} G(|D_s^A u(x,y)|)\frac{dy}{|x-y|^n}\right)dx.
\end{align*}

Let us now fix $\varepsilon>0$. By the monotonicity and convexity of $G$, we can split the previous integral as
\begin{equation}\label{1}
\begin{split}
s \, I_{s,G}^A(u) 
&=2s  \, \int_{\mathbb{R}^n}\left(\int_{\{|y|\geq 2|x|\}} G(|D_s^A u(x,y)|)\frac{dy}{|x-y|^n}\right)dx\\&\quad+2s\, 
\int_{\mathbb{R}^n}\left(\int_{\{|x|\leq |y|<2|x|\}} G(|D_s^A u(x,y)|)\frac{dy}{|x-y|^n}\right)dx\\
&\leq \frac{2s}{1+\varepsilon}J_1 + \frac{2s\varepsilon}{1+\varepsilon}J_2 + 2s J_3
\end{split}
\end{equation}
where
\begin{align*}
J_1&:=\int_{\mathbb{R}^n}\left(\int_{\{|y|\geq 2|x|\}} G\left((1+\varepsilon)\frac{|u(x)|}{|x-y|^s}\right)\frac{dy}{|x-y|^n}\right)dx\\
J_2&:=\int_{\mathbb{R}^n}\left(\int_{\{|y|\geq 2|x|\}} G\left(\frac{1+\varepsilon}{\varepsilon}\frac{|u(y)|}{|x-y|^s}\right)\frac{dy}{|x-y|^n}\right)dx\\
J_3&:= 
\int_{\mathbb{R}^n}\left(\int_{\{|x|\leq |y|<2|x|\}} G(|D_s^A u(x,y)|)\frac{dy}{|x-y|^n}\right)dx.
\end{align*}

Taking into account \cite[(2.22) and (2.23)]{ACPS}, it holds that
\begin{equation}\label{2}
    J_1\leq \frac{w_n}{n s}\int_{\mathbb{R}^n} \overline{G}\left((1+\varepsilon) \frac{|u(x)|}{|x|^s}\right)\,dx
\end{equation}
and
\begin{equation}\label{3}
    J_2\leq w_n \int_{\mathbb{R}^n} G\left(\frac{1+\varepsilon}{\varepsilon} 2^s \frac{|u(y)|}{|y|^s} \right)\,dy.
\end{equation}
To conclude the proof, we only need to provide a suitable upper estimate from above for $J_3$. 
We claim that, fixed $N>3$, there exists $\overline{s}\in(0,1)$ such that
\begin{equation}\label{4}
    \begin{split}
        J_3 \leq  
        \iint_{\mathbb{R}^n \times E}
		 G\left(N^{\overline{s}-s}|D_{\overline{s}}^A u(x,y)|
        \right) \,d\mu +\frac{\varepsilon}{s}
    \end{split}
\end{equation}
for any $s\in(0,\overline{s})$, 
where $E:=\{|x|\leq|y|<2|x|,|x-y|\leq N\}$.\\

In order to prove \eqref{4}, we first notice that $J_3$ can be written as
\begin{equation*}
J_3 = (i)+(ii), 
\end{equation*}
where, denoting $F:=\left\{|x|\leq|y|<2|x|,\, |x-y|>N\right\}$, we write
\begin{align*}
(i):=\iint_{\mathbb{R}^n\times E} G(|D_s^A u(x,y)|)\,d\mu, \qquad (ii):=\iint_{\mathbb{R}^n\times F} G(|D_s^A u(x,y)|)\,d\mu.
\end{align*}

Since $u\in\bigcup_{s\in (0,1)}W^{s,G}_{A,0}(\mathbb{R}^n;\mathbb{C})$, there exists $\overline{s}\in(0,1)$ such that $u\in W^{\overline{s},G}_{A,0}(\mathbb{R}^n;\mathbb{C})$. Let now $s<\overline{s}$. Then
\begin{align}\label{ineq}
\begin{split}
    (i)&=\iint_{\mathbb{R}^n \times E}  G\left(|x-y|^{\overline{s}-s}|D_{\overline{s}}^A u(x,y)|\right)\,d\mu \leq \iint_{\mathbb{R}^n \times E}  G\left(N^{\overline{s}-s}|D_{\overline{s}}^A u(x,y)|\right)\,d\mu.
\end{split}
\end{align}
Moreover, since $|x|\leq|y|<2|x|$ and $|x-y|>N$ implies that
\begin{align}\label{5}
    |x|>\frac{N}{3}\text{ and}\ |y|>\frac{N}{3},
\end{align}
then, by \eqref{5}, the monotonicity and the convexity of $G$, by a change of variable and taking into account that $\Big|e^{i(x-y) A\left(\frac{x+y}{2}\right)}\Big|=1$, we get
\begin{align*}
    (ii)
    &\leq \frac12 \iint_{\mathbb{R}^n \times F} G\left(\frac{2|u(x)|}{|x-y|^s}\right)\,d\mu + \frac12 \iint_{\mathbb{R}^n \times F} G\left(\frac{2|u(y)|}{|x-y|^s}\right)\,d\mu\\
    &\leq\frac12\int_{\{|x|>\frac{N}{3}\}}\left(\int_{\{|x-y|>N\}}G\left(\frac{2|u(x)|}{|x-y|^s}\right)\,\frac{dy}{|x-y|^n}\right)
    \,dx\\
    &\quad+\frac12\int_{\{|y|>\frac{N}{3}\}}\left(\int_{\{|x-y|>N\}}G\left(\frac{2|u(y)|}{|x-y|^s}\right)\,\frac{dx}{|x-y|^n}\right)
    \,dy\\
    &=\int_{\{|x|>\frac{N}{3}\}}\left(\int_{\{|x-y|>N\}}G\left(\frac{2|u(x)|}{|x-y|^s}\right)\,\frac{dy}{|x-y|^n}\right)
    \,dx.
\end{align*}
Finally, the last inequality and a change of variables gives that
\begin{align*}
(ii)&\leq 
\frac{w_n}{n}\int_{\{|x|>\frac{N}{3}\}}\left(\int_N^{+\infty}G\left(\frac{2|u(x)|}{r^s}\right)\,\frac{dr}{r}\right)
    \,dx\\
    &=\frac{w_n}{ns}\int_{\{|x|>\frac{N}{3}\}}\left(\int_0^{\frac{2|u(x)|}{N^s}}G(\tau)\frac{d\tau}{\tau}\right)\,dx\\
    &=\frac{w_n}{ns}\int_{\{|x|>\frac{N}{3}\}}\overline{G}\left(\frac{2|u(x)|}{N^s}\right)\,dx\\
    &<\frac{w_n}{ns}\int_{\{|x|>\frac{N}{3}\}}\overline{G}(2|u(x)|)\,dx<\frac{\varepsilon}{s}
\end{align*}    
for $N$ sufficiently large. From this, \eqref{4} easily follows.\\
\medskip

In order to justify the passage to the limsup, we can argue as in \cite{ACPS} once again. 
In this way, by the arbitrariness of $\varepsilon$, 
by gathering \eqref{1}, \eqref{2}, \eqref{3} and \eqref{4}, by \eqref{overlinehardy} and by using Fatou's Lemma, we close the proof.
\end{proof}
\begin{proof}[Proof of Theorem \ref{MS1}]
Combining the lower bounds obtained in Lemma \ref{thliminf} 
with the upper bounds provided by Lemma \ref{thlimsup}, we easily get \eqref{MainRes}.
\end{proof}


\end{document}